\documentclass[12pt]{amsart}
\usepackage{enumerate}
\usepackage{bbm}
\usepackage{amsmath,amssymb,latexsym}
\newtheorem{theorem}{Theorem}
\newtheorem{lemma}[theorem]{Lemma}
\newtheorem*{lemma*}{Lemma}

\newtheorem*{corollary*}{Corollary}

\newtheorem*{definition*}{Definition}
\newtheorem*{theorem*}{Theorem}

\allowdisplaybreaks
\def\zet{\mathbb{Z}}
\def\ham{\mathbb H_M}
\def\hamp{\mathbb H_M^{+}}
\def\hamm{\mathbb H_M^{-}}
\def\hams{{\mathcal H}_s}
\newcommand{\twoline}[2]{\genfrac{}{}{0pt}{}{#1}{#2}}

\begin{document}

\baselineskip=17pt

\title[ ]{On inverses of discrete rough Hilbert transforms}

\author[M. Paluszynski]{Maciej Paluszynski}
\address{Mathematical Institute\\ Wroc\l aw University\\
pl. Grunwaldzki 2/4\\
50-384 Wroc\l aw, Poland}
\email{mpal@math.uni.wroc.pl}

\author[J. Zienkiewicz]{Jacek Zienkiewicz}
\address{Mathematical Institute\\ Wroc\l aw University\\
pl. Grunwaldzki 2/4\\
50-384 Wroc\l aw, Poland}
\email{zenek@math.uni.wroc.pl}

\date{\today}

\begin{abstract} We describe the structure of the resolvent of the discrete rough truncated Hilbert transform under the critical
exponent. This extends the results obtained in \cite{PZ}
\end{abstract}

\thanks{  }
\maketitle
\section{Introduction and Statement of the Results}
Consider $M\in\mathbb N,\ 1\le\alpha\le1+\frac1{1000},\ 0<\theta<1$ and let $\phi_s$ be appropriate cutoff functions, $\phi_s\in C^\infty_c([\frac12,2])$, uniformly in $s$. We are interested in the truncated Hilbert transform
\begin{align}\label{BT:13}
\ham f(x)&=\sum_{\twoline{M^\theta\le s\le M}{s-\text{dyadic}}}
\sum_{m\ge1}\phi_s\Big(\frac{m^\alpha}{s}\Big)\frac{f(x-[m^\alpha])-f(x+[m^\alpha])}{m}\\
&=\sum_{\twoline{M^\theta\le s\le M}{s-\text{dyadic}}}\hams f(x),\notag
\end{align}
where
\begin{equation*}
\hams f(x)=\sum_{m\ge1}\phi_s\Big(\frac{m^\alpha}{s}\Big)\frac{f(x-[m^\alpha])-f(x+[m^\alpha])}{m},\quad f\in\ell^2(\mathbb Z).
\end{equation*}
We would like to say a few words about our motivation for the study of operators of the above type. The importance of the study of the classical Hilbert transform is evident. It influences many fields like PDEs, $\psi$dOs and function space theory. The widely ranging study of the Hilbert transform, in its more and less exotic variants has led to fundamental questions in many areas. We note that the study of the discrete Hilbert transforms relates to natural questions in ergodic theory.
Operators of the form above and similar have been studied in the past (\cite {B}, \cite{BM}, \cite{C2}, \cite{IW}, \cite{LV}, \cite{MSW} to name only a few). In \cite{PZ} the following theorem has been proved.
\begin{theorem}[\cite{PZ}, Theorem 2.2]\label{PZt}
If $\theta>\alpha-1$ and the operators $(\lambda \mathbb I+\ham)^{-1}$ are bounded on $\ell^2(\mathbb Z)$ uniformly in $M$, then they are all of weak type $(1,1)$, also uniformly in $M$.
\end{theorem}
It has also been proved in \cite{PZ} that $\alpha-1$ is a critical value, that is if $\theta<\alpha-1$ the above theorem does not hold. In fact,
the following theorem has been proved.
\begin{theorem}[\cite{PZ}, Theorem 2.3]
Let  $\theta<\alpha-1$. There exists a sequence of functions $\phi_s$  and a compact set $\Gamma
\subset \mathbb C$
 such that the corresponding Hilbert transform \eqref{BT:13} satisfies
$\|(\lambda\,\mathbb I+\ham)^{-1}\|_{\ell^2\to\ell^2}\le C_I$ for all $M$ and $\lambda \in \Gamma$, and for any $C$ the
estimate $\|(\lambda\,\mathbb I+\,\ham)^{-1}\|_{\ell^1 \rightarrow \ell^{1,\infty}}\le C $,  does
not hold uniformly in $\lambda \in \Gamma$ and $M$.
\end{theorem}

It is the aim of the current note to prove failure of the uniform in $M$ weak type $(1,1)$ estimate
for $(\lambda \mathbb I+\ham)^{-1}$, for any single $\lambda$ satisfying assumption of Theorem
\ref{PZt}. We show that the kernel of $(\lambda \mathbb I+\ham)^{-1}$ has the asymptotic expansion
with the main singular term being the $\ham^2$. Using results from \cite{PZ} Section 5, we know,
that $\ham^2$ do not have a uniform in $M$ weak type (1,1) estimate. The operators $\ham$ do have a
uniform in $M$ weak type $(1,1)$ bound, for any $\theta$. A sketch of the proof of this fact, which
was already mentioned in \cite{PZ}, is included in the Appendix. We thus conclude that the uniform
in $M$ weak type (1,1) estimate for  $(\lambda \mathbb I+\ham)^{-1}$ fails.

Our argument is in fact an elaborated variant of that from \cite{PZ}.
Throughout this work we consider operators $\ham$ of a particular form, similar to these from Theorem 2.3 in \cite{PZ}. We note, however,
that our arguments are flexible and most likely apply to more general $\ham$.

We state necessary definitions. From now on $\theta=\alpha-1-\delta$, where $0<\delta<1$ is small enough, so that $\frac{\alpha-1}{\alpha}+\delta<\alpha-1-\delta$. We put
\begin{equation*}
\mathbb P_M^-=\big[M^{(\alpha-1-\delta)},M^{(\alpha-1)}\big]\cap\mathbb Z,\quad\mathbb P_M^+=\big[M^{(1-\delta)},M\big]\cap\mathbb Z.
\end{equation*}
Finally, let
\begin{equation*}
\ham f(x)=\hamm f(x)+\hamp f(x)=\sum_{\twoline{s\in\mathbb P_M^-}{s-\text{dyad.}}}\hams f(x)+\sum_{\twoline{s\in\mathbb P_M^+}{s-\text{dyad.}}}\hams f(x).
\end{equation*}
This will be the fixed operator for the remainder of this paper. We do not specify the particular functions $\phi_s$ involved, but that should be obvious. Each $\hams$ is called a transform block of scale $s$, and thus $\ham$ is a sum of transform blocks of scales just below $M$ and just below $M^{\alpha-1}$.
Let us adopt some notation from \cite{PZ}. For a dyadic $s$ we say that the kernel $K$ is a ``CZ building block of scale $s$" if, for some constants $D,\,\omega>0$, we have:
\begin{enumerate}[$(i)_s$]
\item  $\sum_x K(x)=0$
\item $\text{supp }K\subset[-s,s]$
\item $\sum_x|K(x)|^2\le\frac{D^2}{s}$
\item $\sum_x|K(x+h)-K(x)|^2\le\frac{D^2}{s}\Big(\frac{|h|}{s}\Big)^{\omega}$
\end{enumerate}
The constant $D$ is particular to a kernel, while $\omega$ will be universal (depending, possibly,
on $\alpha$) throughout this paper. We do not specify it here, requirements imposed on its size
will appear as necessary. Note that a CZ block of scale $s$ is also a CZ block of scale $2s$, with
a larger constant:
\begin{equation*}
D_{2s}=2^{1/2+\gamma}\cdot D_s.
\end{equation*}
When referring to a CZ building block we always assume that its scale is a dyadic integer, so we will often abbreviate
\begin{equation*}
  \sum_{s-\text{dyadic}}K_s\quad\text{as}\quad\sum_sK_s,
\end{equation*}
and similarly.
Occasionally we will use building blocks that will satisfy $(ii)_s,\ (iii)_s$ and $(iv)_s$ but not $(i)_s$. In each such situation we will make clear what the assumptions are.
If $K=\sum_{s-\text{dyad.}}K_s$, and each $K_s$ is a CZ building block of scale $s$, then we call
$K$ a CZ kernel and $\|K\|_{CZ}$ is the maximum of $D$ constants of all blocks $K_s$ (and then the
infimum over all representations of $K$ as a dyadic sum of CZ building blocks). The goal of the
current note is the following (recall, that the operators $\ham$ are of a fixed form, with particular choice of cutoff functions $\phi_s$ described above).
\begin{theorem}\label{ISR:2}
Suppose that for some $\lambda\in\mathbb C$ the operators $(\lambda \mathbb I+\ham)^{-1}$ are defined and bounded on $\ell^2(\mathbb Z)$ uniformly in $M\ge M_0$. Then these operators have the form
\begin{equation*}
\lambda_M'\mathbb I+\beta_M\ham+\gamma_M\ham^2+K^M,
\end{equation*}
where
\begin{equation}\label{ISR:1}
K^M=\sum_{\twoline{s\ge M^{(\alpha-1-\delta)}}{s-\text{dyad.}}}K_s,
\end{equation}
with each $K_s$ a CZ building block of scale $s$ and $\lambda'_M,\beta_M,\gamma_M$ and $\|K^M\|_{CZ}$ bounded uniformly in $M\ge M_1$, with $\gamma_M$ uniformly bounded away from 0, for some $M_1$ depending on $M_0$ and $\lambda$.
\end{theorem}

As an immediate corollary we obtain Theorem 2.5 of \cite{PZ}.

This note is a follow-up to \cite{PZ}, and is not entirely self-contained. The knowledge of
\cite{PZ} is often necessary. Also, for a more complete list of references see \cite{PZ}.

\section{Basic Tools}
In what follows we use $\ham$ to denote both the operators as well as their kernels. This should not cause confusion. Let us recall, that CZ building blocks always have indices (which indicate their scales) that are dyadic integers. Similarly as in \cite{PZ} we use suitably defined Banach algebras.
\begin{definition*}\label{ISR:3}
Let $A_M$ be a normed vector space of convolution operators with kernels of the form
\begin{equation}\label{BT:1}
T_M=\lambda\cdot\delta_0+\beta\cdot\ham +\gamma\cdot\ham^2+K^M,
\end{equation}
where  $\ham$ is given by \eqref{BT:13} and $K^M$ is a CZ kernel of the form \eqref{ISR:1}. Let
\begin{equation}\label{BT:2}
\|T_M\|_{A_M}=\inf\{|\lambda|+|\beta|+|\gamma|+\|K^M\|_{CZ}\}
\end{equation}
where the infimum is taken over all representations of $T_M$ in the form \eqref{BT:1}.
\end{definition*}
We state the following two simple lemmas. These lemmas seem to be folklore, we include the proofs for the sake of completeness.
\begin{lemma}\label{CS}
Suppose $K$, $L$ are CZ kernels as defined above, that is each being a dyadic sum of CZ building blocks of scales $\ge \mathcal J$. Then their convolution is also a CZ kernel of the same form. Moreover
\begin{equation*}
  \|K*L\|_{CZ}\le C\|K\|_{CZ}\|L\|_{CZ}.
\end{equation*}
The constant $C$ is independent of $\mathcal J$.
\end{lemma}
\begin{proof}
Write
\begin{equation*}
K=\sum_{s\ge \mathcal J}K_s,\qquad L=\sum_{s\ge \mathcal J}L_s,
\end{equation*}
the representations as CZ building blocks. Then
\begin{equation*}
K*L=\sum_{s,s'\ge \mathcal J}K_s*L_{s'}=\sum_{s\ge \mathcal J}K_s*\sum_{\mathcal J\le s'\le s}L_{s'}+\sum_{s\ge \mathcal J}L_s*\sum_{\mathcal J\le s'<s}K_{s'}.
\end{equation*}
Observe that $K_s*\sum_{s'\le s}L_{s'}$ is a CZ building block of scale $2s$ satisfying  $(iii)_{2s}$ with constant $D$ satisfying
\begin{align*}
D&\le\sqrt{2s}\|K_s\|_{\ell^2}\cdot\|\sum_{s'\le s}L_{s'}\|_{\ell^2\to\ell^2}\\
&\le C\|K\|_{CZ}\cdot\|L\|_{\ell^2\to\ell^2}\\
&\le C'\|K\|_{CZ}\cdot\|L\|_{CZ}.
\end{align*}
The next to the last inequality follows from Lemma \ref{BT:5}. The second summand in $K*L$ is treated similarly. The same argument applied to the kernel $K(x+h)-K(x)$ yields the estimate of the constant in $(iv)_s$.
\end{proof}
\begin{lemma}\label{tele}\label{teltel}
Suppose
\begin{equation}\label{no_can_0}
K=\sum_{\mathcal J\le s\le\mathcal M}K_s,\qquad\|K\|_{\ell^2\to\ell^2}\le1,
\end{equation}
where $K_s$ are CZ building blocks of scale $s$, possibly without vanishing means (property $(i)_s$), with constants $D_s\le1$. Then
$K$ can be written as
\begin{equation}\label{no_can_1}
  K=\sum_{\mathcal J\le s}{\tilde K}_s,
\end{equation}
where ${\tilde K}_s$ are CZ building blocks (with vanishing means), and $\|K\|_{CZ}\le C$ for some universal constant $C$.
Moreover,
\begin{equation}\label{no_can_3}
  K=\sum_{\mathcal J\le s\le\mathcal M}{\tilde K}_s+{\tilde {\tilde K}}_{4\mathcal M},
\end{equation}
where ${\tilde {\tilde K}}_{4\mathcal M}$ is a CZ building block, possibly without vanishing means, of scale $4\mathcal M$, with constant bounded by a universal constant $C$.
\end{lemma}
\begin{proof}

Let $|x|\le 2s$. Denote by  $\mathbbm{1}_{s}$  an indicator function of $(-s,s)$. It is evident by support analysis, that for $j\le s$ we have $\mathbbm{1}_{4s}*K_j(x)=\sum_{y}K_j(y)$.
Consequently, directly from the definition of $K$,
\begin{align}
\Big|\mathbbm{1}_{4s}*K(x)-\sum_{j\le s}\sum_{y}K_j(y)\Big|&\le\sum_{s<j}\sum_{y}|K_j(y)|\mathbbm{1}_{4s}(x-y)\notag\\
&\le \sum_{s<j} \|\mathbbm{1}_{4s}\|_{\ell^2} \|K_j\|_{\ell^2}\label{BT:18}\\
&\le (8s)^{\frac12}\sum_{s<j}\frac{1}{(j)^{\frac12}}\notag\\
&=\frac{8^{\frac12}}{\sqrt2-1}=E\notag
\end{align}
with $E$ being a universal constant.
Hence we obtain
\begin{align}\label{can_nocan_av}
\big|4s\sum_{j\le s}\sum_{y}K_j(y)\big|&\le \big |\langle\mathbbm{1}_{4s}*K,\mathbbm{1}_{2s} \rangle\big | +4Es\notag\\
&\le \sqrt{32}s\|K\|_{\ell^2\rightarrow \ell^2}+4Es\\
&\le Cs\notag
\end{align}
where  $C$ is a universal constant.

Let $\psi\in C_c^\infty$ be supported in $[-2,2]$,
$\psi\equiv1$ on  $[-1,1]$.  Let $\psi_{s}(x)=\psi(\frac{x}{s})$, and let
\begin{equation*}
  c_s=\sum_y \psi_{s}(y),\quad \tilde\psi_{s}(x)=\frac1{c_s}\psi_{s}(x),\text{ for } s\ge\mathcal J,\quad\tilde\psi_{\mathcal J/2}(x)\equiv0.
\end{equation*}
We write the telescoping identity
\begin{equation*}
K(x)=\sum_{\mathcal J\le s}\Big( K_s(x)-\tilde\psi_{s}(x)\sum_{\mathcal J\le j\le s}\sum_{y}K_j(y)+
\tilde\psi_{s/2}(x)\sum_{\mathcal J\le j<s}\sum_{y}K_j(y)\Big)
\end{equation*}
and denote
$$
\tilde K_s(x)=
 K_s(x)-\tilde\psi_{s}(x)\sum_{j\le s}\sum_{y}K_j(y)+
\tilde\psi_{s/2}(x)\sum_{j<s}\sum_{y}K_j(y)
$$
Obviously $\tilde K_s$ is of mean zero.  Hence, $\tilde K_s(x)$ is a $CZ$ block  due to the fact
 that both $K_s$ and $\psi_s$ satisfy axioms $(ii)_s, (iii)_s, (iv)_s$, and  we have estimate \eqref{can_nocan_av}.
Finally, observe that \eqref{no_can_3} follows immediately from \eqref{no_can_1}.
We multiply \eqref{no_can_1} by $\psi_{2\mathcal M}$ and obtain
\begin{align*}
K(x)&=K(x)\psi_{2\mathcal M}(x)\\
&=\sum_{\mathcal J\le s}{\tilde K}_s(x)\psi_{2\mathcal M}(x)\\
&=\sum_{\mathcal J\le s\le\mathcal M}{\tilde K}_s(x)+\sum_{2\mathcal M\le s}{\tilde K}_s(x)\psi_{2\mathcal M}(x),
\end{align*}
Now, by $(iii)_s, (iv)_s$ and Cauchy--Schwarz inequality,  the second sum represents the  $CZ$ block of scale $4\mathcal M$, where the argument is similar to \eqref{BT:18}.
\end{proof}
The basic tool of this paper is the following submultiplicative inequality, which generalizes the ones in \cite{C1} and \cite{PZ}.
\begin{theorem}\label{BT:3}
If $T_M^1,T_M^2,T_M^3,T_M^4$ are operators of the form \eqref{BT:1} then
\begin{align}\label{BT:4}
&\|T_M^1\cdot T_M^2\cdot T_M^3\cdot T_M^4\|_{A_M} \notag\\
&\quad\le C\sum_{\sigma-\text{permut.}}\|T_M^{\sigma(1)}\|_{A_M}\cdot\|T_M^{\sigma(2)}\|_{A_M}\cdot\|T_M^{\sigma(3)}\|_{A_M}\cdot\|T_M^{\sigma(4)}\|_{\ell^2\to\ell^2}+\\
&\qquad+\epsilon(M)\cdot\|T_M^1\|_{A_M}\cdot\|T_M^2\|_{A_M}\cdot\|T_M^3\|_{A_M}\cdot\|T_M^4\|_{A_M}\notag
\end{align}
where $C$ and $\epsilon(M)$ are independent of the operators $T_M^i$ and $\epsilon(M)\to0$ as $M\to\infty$.
\end{theorem}
In particular, there is $C>0$ independent of $M$, so that $C\|\cdot\|_{A_M}$ is submultiplicative.
\begin{corollary*}
If a family of operators $\lambda\mathbb I +\ham$ satisfy invertibility requirements of Theorem \ref{ISR:2} then the inverse operators satisfy
\begin{equation*}
\big\|(\lambda\mathbb I+\ham)^{-1}\big\|_{A_M}\le C,\qquad M\ge M_1,
\end{equation*}
for some $M_1$.
\end{corollary*}
\begin{proof} The proof follows the same lines as the proof of Theorem 2.5 in \cite{PZ}. The key is an inductive proof of the fact that, with a suitable choice of $N_0$ and $M_1$, large enough, the norm $\|T_M^{4^{N_0}}\|_{A_M}$, $M\ge M_1$ can be made arbitrarily small. In \cite{PZ}, power $2^{N_0}$ was used instead of $4^{N_0}$. Other details of the proof are the same as in \cite{PZ}.
\end{proof}
We will use the following cutoff functions. Let $w\in C^\infty(|x|\le2),\,w\equiv1$ on $|x|\le1$, and $\tilde w(x)=w(x)-w(2 x)$. Thus $\tilde w\in C^\infty(1/2\le|x|\le2)$. For each dyadic $s\ge M^\theta$ we let
\begin{align*}
\varphi_s(x)&=w\Big(\frac{x}{s}\Big),\quad\text{first dyadic $s\ge M^\theta$}\\
\varphi_s(x)&=\tilde w\Big(\frac{x}{s}\Big),\quad\text{remaining $s\ge M^\theta$}
\end{align*}
Suppose $T$ is an operator kernel. We write $T_s(x)=T(x)\,\varphi_s(x)$ and
\begin{equation*}
\mathbb T_s(x)=\sum_{\twoline{M^\theta\le j\le s}{j-\text{dyadic}}} T_j(x).
\end{equation*}

\begin{proof}[Proof of Theorem \ref{BT:3}] The proof is a more careful version of the argument in \cite{PZ}. We begin by formulating basic facts. In what follows we will often encounter terms of the form $\varphi_s\cdot\ham$. We will treat them as equal to $\mathcal H_s$ (despite the fact that some of the defining functions $\phi_s$ are different), which simplifies the write-up and does not affect the line of argument. This is so, because whenever we will make such simplifications, we will not need particular properties of fixed functions $\phi_s$.
\begin{lemma}\label{BT:9}
For a suitable choice of a constant $C>0$, $A_M$ with norm $C\|\cdot\|_{A_M}$ is a Banach algebra.
\end{lemma}
\begin{proof} We are going to show that the product
\begin{equation*}
(\lambda\delta_0+\beta\ham+\gamma\ham^2+K^M)\cdot(\tilde\lambda\delta_0+\tilde\beta\ham+\tilde\gamma\ham^2+\tilde K^M)
\end{equation*}
belongs to $A_M$. It follows immediately from the definition of the algebra $A_M$ that we need uniform in $M$ estimates for $A_M$ norms of the operators: $\ham^3,\,\ham^4,\,\ham*K^M\ \text{ and }\ham^2*K^M$. In fact we will prove that all of these products are CZ kernels. To do this, we use the following lemma, essentially proved in \cite{PZ}.

\begin{lemma}[\cite{PZ}]\label{Conv} Recall that
 $\mathcal H_{s}$ denotes a transform
 block of scale $s$ and $K_s$  a CZ building block of scale $s$. We have

\begin{enumerate}[(i)]
\item\label{BT:14} $\mathcal H_{s_1}*\sum_{s\in \mathcal{A}}K_{s}$ is a CZ building block of scale $4s_1$,
provided  all $s\in \mathcal{A}$ satisfy $s_1^{\frac{\alpha-1}{\alpha}+\delta}\le s\le s_1$,
\item\label{BT:15}  $\mathcal H_{s_1}*\sum_{s\in \mathcal{A}}K_{s}$ is a CZ kernel, provided
 all $s\in \mathcal{A}$ satisfy $s\ge s_1$,
\item\label{BT:16} $\mathcal H_{s_1}*\sum_{s\in \mathcal{A}}\mathcal H_{s}$ is a CZ building block of scale $4s_1$, provided  all $s\in \mathcal{A}$ satisfy $s_1^{\alpha-1+\delta}\le s\le s_1$.
\end{enumerate}
\end{lemma}
\begin{proof}Part \eqref{BT:16} follows from \cite{PZ}, Lemma 3.8, with
\begin{equation*}
  \tilde{\mathbb T}_s=\sum_{s\in \mathcal{A}}\mathcal H_{s}.
\end{equation*}
Part \eqref{BT:15} is obvious, since each $\mathcal H_{s_1}*K_{s},\ s\in\mathcal{A}$, is a CZ building block of scale $2s$.
Finally, \eqref{BT:14} has been proved in \cite{PZ} under a stronger assumption $s_1\ge s^{\alpha(1-1/\alpha+\delta/\alpha)}$ ($s$ and $s_1$ are reversed in the argument in \cite{PZ}). That argument works without changes provided we prove a strengthened version of Lemma 3.7 from \cite{PZ}. Namely, observe that statements (iii) and (iv) of that lemma hold under a weaker assumption $s_1\ge s^{1-1/\alpha+\delta/\alpha}$. We now outline the necessary changes in the proof of
Lemma 3.7 from \cite{PZ}. Under the notation from the proof of Lemma 3.7 in \cite{PZ} (in particular the definitions of $I$, $II$ and $III$) we have:
\begin{enumerate}[(a)]
\item $|I|$ depends only on $\|K_{s_1}\|_{\ell^1}$. Change of the range of $s_1$ to $s_1\ge s^{1-1/\alpha+\delta\alpha}$ does not affect the estimate.
\item Since $\|K_{s_1}\|_{\ell^2}^2\lesssim 1/s_1$ (compare with $\|\mathcal{H}_{s_1}\|_{\ell^2}^2\simeq1/s_{1}^{1/\alpha}$), the estimate of $|II|$ improves to:
    \begin{equation*}
      |II|\le\frac{C}{s_1\cdot s^{1/\alpha}}\le\frac{C}{s^{1/\alpha}s^{1-1/\alpha+\delta/\alpha}}=\frac{C}{s^{1+\delta/\alpha}}.
    \end{equation*}
\item For the same reason as in (b), we have, for $s_1\ge s^{1-1/\alpha+\delta/\alpha}$
    \begin{equation*}
      |III|\le\frac{C\cdot s^{1-1/\alpha+\delta/2\alpha}}{s\cdot s_1}=\frac{C}{s^{1+\delta/2\alpha}}.
    \end{equation*}
\end{enumerate}
Other parts of the proof of \cite{PZ} Lemma 3.7 require no changes.
\end{proof}

We now proceed with estimating types of operator products that arise in the proof of Lemma
\ref{BT:9}
\newline(1) Observe, that $\ham*K^M=\hamm*K^M+\hamp*K^M$. Similarly to \cite{PZ},
we apply Lemma \ref{Conv} to the decomposition
$$
\hamm*K^M=\sum_{s_1\in \mathbb{P}^-_M}\mathcal H_{s_1}*\sum_{M^\theta\le s_2\le s_1}K_{s_2}+
\sum_{M^\theta\le s_2}K_{s_2}*\sum_{s_1< s_2, s_1\in \mathbb{P}^-_M}\mathcal H_{s_1}=I+II
$$
By Lemma \ref{Conv} (\ref{BT:14}) for a fixed $s_1\in \mathbb{P}^-_M$,
the kernel  $\mathcal H_{s_1}*\sum_{M^\theta\le s_2\le s_1}K_{s_2}$
is CZ  building block of scale $4s_1$. We infer that  $I$  is a CZ kernel composed of CZ  building blocks
of scales $\ge 4M^{\alpha-1-\delta}$. We now prove a similar statement for $II$.
Observe that for fixed $s_2$ and $T=\sum_{s_1< s_2, s_1\in \mathbb{P}^-_M}\mathcal H_{s_1}$
we have $\|T\|_{\ell^2\rightarrow\ell^2}\le C$ for a universal constant $C$, $\rm{supp }T \subset [-2s_2, 2s_2]$.
Consequently $K_{s_2}*\sum_{s_1< s_2, s_1\in \mathbb{P}^-_M}\mathcal H_{s_1}$ is a
CZ  building block of scale $4s_2$.

Similarly, the same argument applies to the second summand $\hamp*K^M$ proving that it is a CZ kernel composed of CZ
building blocks of scales $\ge 4M^{1-\delta}$. In this case
 Lemma \ref{Conv} (\ref{BT:14}) is applicable provided $\delta$  satisfies
$\alpha-1 -\delta>\frac{\alpha-1}{\alpha}+\delta$, which we have assumed.
\newline (2) We now
consider $\ham^2*K^M$. Decomposing $\ham=\hamm+\hamp$ we need to consider components of the form
$\hamm*\hamm*K^M$, $\hamm*\hamp*K^M$ and $\hamp*\hamp*K^M$. Observe, that by Lemma \ref{Conv}
(\ref{BT:16}) or directly by \cite{PZ}, $\hamm*\hamm$ is a CZ kernel with buildings blocks of
scales between $4M^{\alpha-1-\delta}$ and $4M^{\alpha-1}$. Consequently, $(\hamm*\hamm)*K^M$ is a
composition of two CZ operators of the form \eqref{ISR:1}. By Lemma \ref{CS} it is again an operator of the form \eqref{ISR:1}, that is a sum of CZ
building blocks of scales $\ge 4M^{\alpha-1-\delta}$. Similar argument holds for $(\hamp*\hamp)*K^M$,
again, if $\delta$ is small enough.
Finally,
recall from the case considered above that
$\hamp*K^M$ is a sum of CZ building blocks $K_s$ of scales $s\ge M^{\alpha-\delta}$. We thus have
\begin{equation*}
  \hamm*\hamp*K^M=\sum_{M^{\alpha-\delta}\le s}\hamm*\tilde K_s,
\end{equation*}
and the conclusion follows (since each $\hamm*K_s,\ s\ge M^{\alpha-\delta}$ is a CZ building block).
\newline (3) We now turn to $\ham^3$. Let
\begin{equation*}
U_1=\hamp*\hamp,\quad U_2=\hamm*\hamm.
\end{equation*}
By considerations in the above case (2) $U_1,U_2$ are CZ kernels with building blocks of scales $s\ge
M^{\alpha-\delta}$ and $M^{\alpha-1-\delta}\le s\le M^{\alpha-\delta}$ respectively. By Lemma \ref{Conv} (\ref{BT:14}) all of the kernels
\begin{equation*}
U_1*\hamm,\ U_1*\hamp,\ U_2*\hamm,\ U_2*\hamp
\end{equation*}
are CZ kernels with building blocks of scales $\ge M^{\alpha-1-\delta}$. This finishes the proof
for $\ham^3$.

We omit similar arguments for $\ham^4$. This concludes the proof of Lemma \ref{BT:9}.
\end{proof}
\begin{lemma}[\cite{PZ}, Lemma 3.2]\label{BT:5}
 If $T$ be an operator on $\ell^2$, and $T_s,\,\mathbb T_s$ are the smooth truncations defined above, then $\|T_s\|_{\ell^2\to\ell^2},\,\|\mathbb T_s\|_{\ell^2\to\ell^2}\le C \|T\|_{\ell^2\to\ell^2}$.
\end{lemma}
\begin{lemma}\label{BT:6}
Suppose an operator $T_M$ has the form \eqref{BT:1}. Then
\begin{equation*}
|\lambda|\le \|T_M\|_{\ell^2\to\ell^2}+\epsilon_1(M)\|T_M\|_{A_M},
\end{equation*}
where $\epsilon_1(M)$ can be chosen of the form $C\,M^{-b}$ where $C,b>0$ (thus $\epsilon_1(M)\to\infty$ as $M\to\infty$), with constants $C,b$ universal.
\end{lemma}
\begin{proof}
 By  \eqref{BT:1} we obtain the identity
 \begin{equation*}
\lambda=\langle T_M\delta_0,\delta_0\rangle-\beta\ham(0)-\gamma\ham^2(0)- K_M(0).
\end{equation*}
Moreover, we have $\ham(0)=0$,
\begin{equation*}
|K_M(0)|\le C\|T_M\|_{A_M}M^{-(\alpha-1-\delta)/2}
\end{equation*}
and
\begin{equation*}
|\ham^2(0)|=\Big|\sum_{\twoline{M^\theta\le s\le M}{s-\text{dyadic}}}
\sum_{m\ge1}\phi_s\Big(\frac{m^\alpha}{s}\Big)^2\frac{1}{m^2}\Big|  \le C_\alpha M^{-(\alpha-1-\delta)/\alpha}.
\end{equation*}
Since $|\gamma|\le \|T_M\|_{A_M} $ and $\alpha-1-\delta>0$, the lemma follows.
\end{proof}
\begin{lemma}\label{BT:7}
For $T^i$, $i=1,2,3,4$ operators on $\ell^2$ we have the following decomposition
\begin{equation*}
T^1\cdot T^2\cdot T^3\cdot T^4=\sum_{\sigma-\text{permut.}}\sum_{s-\text{dyadic}}\sum_{s',s'',s'''}\,T^{\sigma(1)}_s\cdot\mathbb T^{\sigma(2)}_{s'}\cdot\mathbb T^{\sigma(3)}_{s''}\cdot\mathbb T^{\sigma(4)}_{s'''}\cdot \epsilon_{s,s',s'',s'''},
\end{equation*}
where $s',s'',s'''\in\{s/2,s,2s\}$, $\epsilon_{s,s',s'',s'''}\in\{0,\pm1\}$.
\end{lemma}
\begin{proof}
Follows immediately from the inclusion-exclusion principle.
\end{proof}

We make the following observation. Consider
\begin{multline*}
\ham^2(x)\cdot\varphi_s(x)=\hamp*\hamp(x)\cdot\varphi_s(x)+\\
+\hamm*\hamm(x)\cdot\varphi_s(x)
+2\hamp*\hamm(x)\cdot\varphi_s(x).
\end{multline*}
We have that $\hamp*\hamp(x)$ is a CZ kernel (\cite{PZ}) thus $\hamp*\hamp(x)\cdot\varphi_s(x)$ satisfies the conditions for a CZ building block except, possibly, the vanishing mean. By Lemma \ref{BT:5} the operators
\begin{equation*}
\mathbb K_{s_1,s_2}^M=\sum_{s_1\le s\le s_2}(\hamp*\hamp)\cdot\varphi_s
\end{equation*}
have their $\ell^2\to\ell^2$ norms bounded uniformly in $s_1,s_2,M$. Using
Lemma \ref{teltel}
we can write
\begin{equation*}
\mathbb K_{s_1,s_2}^M=\sum_{s\ge s_1}\tilde K_s^M
\end{equation*}
where $\tilde K_s^M$ are CZ building blocks with scales $s$.

Similar argument holds for $(\hamm*\hamm)\cdot\varphi_s$. Moreover,
\begin{equation*}
(\hamp*\hamm)\cdot\varphi_s=({\mathcal H}_s*\hamm)\varphi_s=({\mathcal H}_s\cdot\varphi_s)*\hamm+C_s
\end{equation*}
where the commutator $C_s$ is a CZ building block at scale $s$ (possibly without mean 0), with CZ constant $D_s\le c\,s^{-\gamma}$, for some $\gamma>0$, independent of $\alpha$.  The first equality in the above follows immediately from the support considerations. For the last assertion let us consider the following
\begin{align*}
C_s&=({\mathcal H}_s*\hamm)\varphi_s-({\mathcal H}_s\cdot\varphi_s)*\hamm\\
&=\sum_{y\in\mathbb Z}\mathcal H_s(y)\hamm(\,\cdot\,-y)\big(\varphi_s(\,\cdot\,)-\varphi_s(y)\big),
\end{align*}
where $s\in\mathbb P_M^+$, that is $s\ge M^{1-\delta}$. As a consequence
\begin{align*}
|C_s(x)|&\le\sum_{|y-x|\le M^{\alpha-1}}|\mathcal H_s(y)|\cdot\frac{1}{M^{\alpha-1-\delta}}\cdot\frac{|x-y|}{s}\\
&\le C \frac{1}{s^{1/\alpha}}\cdot\frac{1}{M^{\alpha-1-\delta}}\cdot\frac{M^{\alpha-1}}{s}\\
&=C \frac{M^\delta}{s^{1+1/\alpha}}\\
&\le C \frac{s^{\frac{\delta}{1-\delta}}}{s^{1+1/\alpha}}.
\end{align*}
Since $\rm{supp}\ C_s\subset[-2s,2s]$, we have
\begin{equation}\label{BT:17}
\|C_s\|_{\ell^2}^2\le\frac{C}{s^{1+\gamma}},\qquad\text{for some }\gamma>1/2.
\end{equation}
With similar, obvious, H\"older estimate, $C_s$ satisfies $(ii)_s\dots(iv)_s$ from the definition of a CZ building block.

By \eqref{BT:17}
the operators
\begin{equation*}
\mathbb C_{s_1,s_2}=\sum_{s_1\le s\le s_2} C_{s},
\end{equation*}
have  their $\ell^2\to\ell^2$ norms bounded uniformly in $s_1,s_2, M$. Consequently, by Lemma \ref{teltel}, we can write
down the representation
\begin{equation*}
\mathbb C_{s_1,s_2}=\sum_{s\ge s_1}\tilde  C_{s},
\end{equation*}
where $\tilde  C_{s}$ is a CZ building block. Thus we have proved

\begin{lemma}\label{HHdecomp}
We have the following identity
\begin{equation*}
\sum_{s_1\le s \le s_2}\ham^2(x)\cdot\varphi_s(x)=\hamp\cdot\big(w(\,\cdot\,/s_2)-w(\,2\,\cdot\,/s_1)\big)*\hamm(x)+
\mathbb K_{s_1}^M(x)
\end{equation*}
where $\mathbb K_{s_1}^M$ is a CZ kernel containing building blocks with scales larger than $s_1/2$.
\end{lemma}
Observe, that if $s_2\le M^{1-\delta}/4$ then
\begin{equation*}
\hamp(x)\cdot\Big(w\Big(\frac{x}{s_2}\big)-w\Big(\frac{2\,x}{s_1}\Big)\Big)\equiv0,
\end{equation*}
by support considerations. Otherwise the support of the above product is contained in $[-8s_2,8s_2]$ (recall, that $\rm{supp }\ \hamm\subset[-M^{\alpha-1},M^{\alpha-1}]$, and $\alpha-1<1-\delta$). We can therefore assume that $\mathbb K_{s_1}^M$ is composed of building blocks with scales no greater that $8s_2$, with the last one possibly without vanishing mean.
As a corollary we obtain

\begin{lemma}\label{BT:10}
Assume that operator $T_M$ of the form \eqref{BT:1} does not contain the $\delta_0$ component.
If a CZ kernel $K^M$ contains CZ building blocks only of scales $s$, with $s_1^{\frac{\alpha-1}{\alpha}+\delta}\le s\le s_1$, then $T_{M,s_1}*K^M$ is a CZ building block with scale no greater than $9s_1$.
\end{lemma}
\begin{proof} The truncated operator $T_{M,s_1}$ consists of 3 parts, we are going to consider the component $\ham^2(x)\cdot\varphi_{s_1}(x)$. The other two components are: a CZ block for which the conclusion is obvious, and a truncated $\ham\cdot\varphi_{s_1}$ which can be treated in a similar way. We apply representation from Lemma \ref{HHdecomp} with $s_2=s_1$. By the observation above, the CZ part of the representation consists of CZ building blocks of scales between $s_1/2$ and $8s_1$. These, convolved with $K^M$ produce a building block of scale no greater that $9s_1$. Now,
\begin{equation*}
S(x)=\hamp(x)\cdot\Big(w\Big(\frac{x}{s_2}\big)-w\Big(\frac{2\,x}{s_1}\Big)\Big)
\end{equation*}
contains, essentially, one block and, by Lemma \ref{Conv} (\ref{BT:14}) $S*K_M$ is a CZ block of scale $4s_1$, and consequently $S*K_M*\hamm$ is a building block of scale no greater that $5s_1$.
\end{proof}

We now return to the proof of Theorem \ref{BT:3}. Before we proceed we make the following observation. Without loss of generality we can assume that all four operators appearing in \eqref{BT:4} have no $\delta_0$ component. Let us justify this.
Write
\begin{equation*}
T_M^1=\tilde T_M^1+\lambda_1\delta_0,
\end{equation*}
that is, isolate the $\delta_0$ part from $T_M^1$. We have
\begin{equation*}
T_M^1\cdot T_M^2\cdot T_M^3\cdot T_M^4=\tilde T_M^1\cdot T_M^2\cdot T_M^3\cdot T_M^4+\lambda_1\cdot T_M^2\cdot T_M^3\cdot T_M^4=I+II.
\end{equation*}
$A_M$ is a Banach algebra, so
\begin{align*}
\|II\|_{A_M}&\le |\lambda_1|\,\| T_M^2\|_{A_M}\,\| T_M^3\|_{A_M}\,\| T_M^4\|_{A_M}\\
\text{by Lem. \ref{BT:6}}&\le(\|T_M^1\|_{A_M}\cdot\epsilon(M)+C\| T_M^1\|_{\ell^2\to\ell^2})\times\\
&\qquad\times\| T_M^2\|_{A_M}\,\| T_M^3\|_{A_M}\,\| T_M^4\|_{A_M}.
\end{align*}
Thus for the part $II$ we obtain desired estimate \eqref{BT:4}. It is therefore enough to estimate part $I$, that is with $T_M^1$ without the $\delta_0$ component.
We then iterate for the rest of the operators. That justifies the above observation.

 We proceed with the proof of Theorem \ref{BT:3}. We are actually going to prove that
 \begin{equation*}
   P=T_M^1\cdot T_M^2\cdot T_M^3\cdot T_M^4
 \end{equation*}
is a CZ operator, with the norm $\|P\|_{CZ}$ bounded by the required right hand side of \eqref{BT:4}. We will do so by showing, that $P$ decomposes into CZ building blocks, possibly without vanishing means. Since, clearly,
\begin{equation*}
  \|P\|_{\ell^2\to\ell^2}\le C\,\|P\|_{CZ}\le \text{ right hand side of }\eqref{BT:4}
\end{equation*}
the  Lemma \ref{tele} applies and the proof would be finished.

Using Lemma \ref{BT:7} we investigate operators of the form
\begin{equation*}
\tilde Z=T_{M,s_1}^1*\mathbb T_{M,s_2}^2* \mathbb T_{M,s_3}^3*\mathbb T_{M,s_4}^4.
\end{equation*}
We will show that $\tilde Z$ is a CZ building block, possibly without vanishing means, with
\begin{equation}\label{ztilde}
\|\tilde Z\|_{CZ}\le C\cdot\|T^1_M\|_{A_M}\cdot\|T^2_M\|_{A_M}\cdot\|T^3_M\|_{A_M}\cdot\|T^4_M\|_{\ell^2\to\ell^2}.
\end{equation}

It will suffice to prove that
\begin{equation*}
Z=T_{M,s_1}^1*\mathbb T_{M,s_2}^2*\mathbb T_{M,s_3}^3
\end{equation*}
is a CZ building block of scale $s_1$, possibly without vanishing means, with
\begin{equation*}
\|Z\|_{CZ}\le C\cdot\|T^1_M\|_{A_M}\cdot\|T^2_M\|_{A_M}\cdot\|T^3_M\|_{A_M}
\end{equation*}
This would imply \eqref{ztilde}
\begin{align*}
\|\tilde Z\|_{CZ}&\le C\|Z\|_{CZ}\cdot\|T^4_{M,s_4}\|_{\ell^2\to\ell^2}\\
&\le C\cdot\|T^1_M\|_{A_M}\cdot\|T^2_M\|_{A_M}\cdot\|T^3_M\|_{A_M}\cdot\|T^4_M\|_{\ell^2\to\ell^2}.
\end{align*}

For the operators $\mathbb T_{M,s_2}^2$ and $\mathbb T_{M,s_3}^3$, we use representation given by Lemma \ref{HHdecomp}, namely for $i=2,3$
\begin{align*}
\mathbb T_{M,s_i}^i&=\sum_{M^\theta\le s\le s_i}\tilde{\mathcal H}_s+\sum_{M^\theta\le s\le s_i}\tilde{\tilde{\mathcal H}}_s^+*\hamm+K^{M,i}\\
&=\tilde {\mathbb H}_{M,s_i}+\tilde{\tilde{\mathbb H}}_{M,s_i}^+*\hamm+K^{M,i},
\end{align*}
where $K^{M,i}$ satisfy conditions from the observation stated after Lemma \ref{HHdecomp}. $\tilde{\mathcal H}_s=\ham\cdot\varphi_s$ and  $\tilde{\tilde{\mathcal H}}_s^+=\ham^+\cdot\varphi_s$. Tildes underscore the fact that, due to truncation, the Hilbert transform blocks have different cutoff functions $\varphi$ (but with all other properties preserved). In particular $\tilde{\tilde{\mathcal H}}_{s_i}^+$ vanish for $s_i<M^{1-\delta}/4$.
Consider the operator
\begin{align*}
&\big(\tilde {\mathbb H}_{M,s_2}+\tilde{\tilde{\mathbb H}}_{M,s_2}^+*\hamm+K^{M,2}\big)*\big(\tilde {\mathbb H}_{M,s_3}+\tilde{\tilde{\mathbb H}}_{M,s_3}^+*\hamm+K^{M,3}\big)\\
&\quad=\tilde {\mathbb H}_{M,s_2}*\tilde {\mathbb H}_{M,s_3}+
\big(\tilde {\mathbb H}_{M,s_2}*\tilde{\tilde{\mathbb H}}_{M,s_3}^+*\hamm+\tilde {\mathbb H}_{M,s_3}*\tilde{\tilde{\mathbb H}}_{M,s_2}^+*\hamm\big)+\\
&\qquad+\tilde{\tilde{\mathbb H}}_{M,s_2}^+*\tilde{\tilde{\mathbb H}}_{M,s_3}^+*\hamm*\hamm+\text{   terms containing $K^{M,i}$}\\
&\quad= I+II+III+IV.
\end{align*}

We now decompose
\begin{equation*}
  \tilde {\mathbb H}_{M,s_3}=\tilde {\mathbb H}_{M,s_3}^-+\tilde {\mathbb H}_{M,s_3}^+,
\end{equation*}
(grouping transform blocks of scales from $\mathbb P_M^-$ and $\mathbb P_M^+$). Then $\hamm*\tilde {\mathbb H}_{M,s_3}^-$ is a CZ kernel satisfying the assumptions of Lemma \ref{BT:10}. Similarly $\tilde{\tilde{\mathbb H}}_{M,s_3}^+*\tilde {\mathbb H}_{M,s_2}$ is also a CZ kernel satisfying the assumptions of that lemma. Consequently,
$T_{M,s_1}^1*\,II$ is a CZ block of scale $s_1$. Similar argument works for $T_{M,s_1}^1*\,III$, because of the factor $\hamm*\hamm$ which is a CZ kernel satisfying the assumptions of Lemma \ref{Conv} (\ref{BT:14}). We are left with the operator
\begin{equation*}
  P=T_{M,s_1}^1*\tilde {\mathbb H}_{M,s_2}*\tilde {\mathbb H}_{M,s_3},
\end{equation*}
and we want to prove that it is a CZ building block of scale $2s_1$. We apply Lemma \ref{HHdecomp} to $T_{M,s_1}^1$. We get
\begin{align*}
  P &=\tilde {\mathbb H}_{M,s_2}*\tilde {\mathbb H}_{M,s_3}* \tilde{\mathcal H}_{s_1}+\tilde {\mathbb H}_{M,s_2}*\tilde {\mathbb H}_{M,s_3}* \tilde{\mathcal H}_{s_1}*\hamm+\\
  &\quad+ \text{ products containing CZ building blocks of scales $s_1$}\\
  &=I_P+II_P+III_P.
\end{align*}
Now, $III_P$ is a CZ building block because of supports, while $I_P,\,II_P$ are appropriate building blocks by argument similar to that in the proof of the fact that $A_M$ is a Banach algebra (Lemma \ref{BT:9}).

To deal with $IV$ it is enough to prove that $T^1_{M,s_1}*K^{M,i}$ is a CZ building block of scale $\le 4s_1$. Applying Lemma \ref{BT:10} to $T^1_{M,s_1}$ we see, that it is sufficient to prove the claim for ${\mathbb H}_M^-*\tilde {\mathbb H}_{M,s_1}^+*K^{M,i}$. From estimates in (2)
of Lemma \ref{BT:9} we know, that $\tilde {\mathbb H}_{M,s_1}^+*K^{M,i}$ is a CZ building block of scale $s_1$. As a consequence we get the estimate in the case $IV$ and the theorem follows.

\end{proof}
We conclude the study of the singularity of the inverse operators with the following theorem.
\begin{theorem}\label{BT:11}
Suppose for some fixed $\lambda,\beta\in\mathbb C$ the operators
\begin{equation*}
\big(\lambda\delta_0+\beta\ham\big)^{-1}
\end{equation*}
are bounded on $\ell^2$, uniformly in $M\ge M_0$. By Theorem \ref{ISR:2} they have the form
\begin{equation*}
\big(\lambda\delta_0+\beta\ham\big)^{-1}=\tilde\lambda_M\delta_0+\tilde\beta_M\ham+\tilde\gamma_M\ham^2+K^M,
\end{equation*}
and are bounded in the algebra $A_M$, uniformly in $M\ge M_0$. In this situation $\tilde\gamma_M$ are bounded away from $0$: $|\tilde\gamma_M|\ge\gamma>0$.
\end{theorem}
\begin{proof}
We have
\begin{align*}
\delta_0&=\big(\lambda\delta_0+\beta\ham\big)\big(\tilde\lambda_M\delta_0+\tilde\beta_M\ham+\tilde\gamma_M\ham^2+K^M\big)\\
&=(\lambda\tilde\lambda_M)\delta_0+(\beta\tilde\lambda_M+\tilde\beta_M\lambda)\ham+(\lambda\tilde\gamma_M+\beta\tilde\beta_M)\ham^2+\\
&\qquad+\ham^3\,(\beta\cdot\tilde\gamma_M)+\lambda K^M+\beta\ham K^M.
\end{align*}
Thus
\begin{equation}\label{BT:12}
(\beta\tilde\lambda_M+\tilde\beta_M\lambda)\ham+2(\lambda\tilde\gamma_M+\beta\tilde\beta_M)(\hamp*\hamm)=(1-\lambda\tilde\lambda)\delta_0+\tilde K^M,
\end{equation}
where $\tilde K^M$ is a CZ operator with $\|\tilde K^M\|_{CZ}\le C$, uniformly in $M\ge M_0$. We have
\begin{multline*}
(1-\lambda\tilde\lambda_M)+\langle\tilde K^M\delta_0,\delta_0\rangle=\\
=(\lambda\tilde\gamma_M+\beta\tilde\beta_M)\langle(\hamp*\hamm)\delta_0,\delta_0\rangle+(\beta\tilde\lambda_M+
\tilde\beta_M\lambda)\langle\ham\delta_0,\delta_0\rangle.
\end{multline*}
Observe
\begin{gather*}
\langle\ham\delta_0,\delta_0\rangle=0,\\
\langle(\hamp*\hamm)\delta_0,\delta_0\rangle=\langle\hamm,\hamp\rangle=0,\\
\langle\tilde K^M\delta_0,\delta_0\rangle\to0,
\end{gather*}
so $1-\lambda\tilde\lambda_M\to0$ as $M\to\infty$. In particular, $\tilde\lambda_M$ are bounded away from $0$. Observe that
$\ham$ and $\hamp*\hamm$ have disjoint supports. Moreover, from equations (5.6) in \cite{PZ} it
follows, that for the left hand side of \eqref{BT:12} to satisfy condition $(iii)_s$ uniformly in
$M$, the following are necessary
\begin{equation*}
\beta\tilde\lambda_M+\tilde\beta_M\lambda\to0\quad\text{and}\quad\lambda\tilde\gamma_M+\beta\tilde\beta_M\to0\quad\text{as}\ M\to\infty,
\end{equation*}
and the conclusion follows.
\end{proof}

\section{Appendix}

We would like to conclude with the following theorem. Let us recall our principal operator \eqref{BT:13}
\begin{equation*}
\ham f(x)=\sum_{\twoline{M^\theta\le s\le M}{s-\text{dyadic}}}
\sum_{m\ge1}\phi_s\Big(\frac{m^\alpha}{s}\Big)\frac{f(x-[m^\alpha])-f(x+[m^\alpha])}{m}.
\end{equation*}
\begin{theorem*} For any $0<\theta<1$ the operator $\ham$ is of weak type $(1,1)$, uniformly in $M$.
\end{theorem*}
\begin{proof}
The theorem can be proved in a way similar to the proof of Theorem 1.1 of \cite{UZ}. We provide the necessary details.
Without loss of generality we can assume that $\theta<1$ is sufficiently
close to 1. (Fix large $A$.  Hilbert transform $\ham$ corresponding to  $\theta$ is a finite sum of Hilbert transforms
${\mathbb  H_{M^{\theta^\frac{j}{A}}}}$, $j\in \{0,1,...,A-1\}$ corresponding to $\theta^\frac{1}{A}$.)

For $f\in\ell^1,\ \lambda>0$ let $\{Q_j\}_j$ be the collection of Calder\'on-Zygmund cubes associated with level $\lambda$, that is the maximal dyadic cubes for which
\begin{equation*}
  \frac{1}{|Q|}\int_Q|f|>\lambda.
\end{equation*}
The collection is clearly pairwise disjoint, and from the definition of the cubes we have
\begin{equation}\label{app:2}
  \sum_j|Q^*_j|\le\frac{2}{\lambda}\cdot\|f\|_{\ell^1},
\end{equation}
where $Q^*_j$ is the immediate dyadic parent of $Q_j$.
We let
\begin{equation*}
  g=f\cdot\mathbbm{1}_{\{(\cup_jQ_j)^c\}},\qquad f_j=f\cdot \mathbbm{1}_{Q_j}.
\end{equation*}
We further let, for $k\in\zet$,
\begin{equation*}
  b_k=\sum_{j:|Q_j|=2^k}f_j,
\end{equation*}
and thus
\begin{equation*}
  f=g+\sum_kb_k,
\end{equation*}
with all summands having disjoint supports. We now further decompose $b_k$'s in the following way. Let
\begin{equation*}
  E_k(f)=\sum_{j:|Q_j|=2^k}\frac{\mathbbm{1}_{Q_j}}{|Q_j|}\int_{Q_j}f,\qquad E_k^{(s)}=E_k(f\cdot\mathbbm{1}_{\{|f|\le\lambda s^{1/\alpha}\}}),
\end{equation*}
and
\begin{equation*}
  b_k^{(s)}=b_k\cdot\mathbbm{1}_{\{|f|>\lambda s^{1/\alpha}\}},\qquad B_k^{(s)}=b_k-b_k^{(s)}-E_k^{(s)}(f).
\end{equation*}
The decomposition is thus as follows, for $s$ dyadic and $k\in\zet$
\begin{equation}\label{app:1}
  f=g+\sum_kb_k^{(s)}+\sum_kB_k^{(s)}+\sum_kE_k^{(s)}(f),
\end{equation}
where
\begin{equation*}
  B_k^{(s)}=b_k\cdot\mathbbm{1}_{\{|f|\le\lambda s^{1/\alpha}\}}-E_k^{(s)}(f).
\end{equation*}
We will consider each of the 4 components of $f$. Consider $g$. Since
\begin{equation*}
  {\rm supp }\,g\subset\Big(\bigcup_jQ_j\Big)^c
\end{equation*}
we have $|g|\le\lambda$. The operators $\ham$ are of strong type (2,2) (with bound uniform in $M$), so by Chebychev's inequality
\begin{equation*}
  \lambda^2\big|\{x:|\ham g(x)|>\lambda\}\big|\le\|\ham g\|_{\ell^2}^2\le C\,\|g\|_{\ell^2}^2\le C\, \lambda\|g\|_{\ell^1}\le C\,\lambda\|f\|_{\ell^1},
\end{equation*}
and consequently
\begin{equation*}
  \big|\{x:|\ham g(x)|>\lambda\}\big|\le \frac{C}{\lambda}\,\|f\|_{\ell^1}.
\end{equation*}
Now, consider $b_k^{(s)}$. By the definition \eqref{BT:13} of $\hams$ we have
\begin{equation*}
  \{x:\hams*b_k^{(s)}(x)\neq0\}\subset\bigcup_{s/2\le r^\alpha\le2s}\big([r^\alpha]+\{x:b_k^{(s)}(x)\neq0\}\big)\cup
  \big(-[r^\alpha]+\{x:b_k^{(s)}(x)\neq0\}\big),
\end{equation*}
so
\begin{align*}
  \big|\{x:\hams*b_k^{(s)}(x)\neq0\}\big| &\le2\sum_{s/2\le r^\alpha\le 2s}\big|\{x:b_k^{(s)}(x)\neq0\}\big| \\
  &\le Cs^{1/\alpha}\big|\{x:|b_k(x)|>\lambda s^{1/\alpha}\}\big|
\end{align*}
Thus
\begin{align*}
  &\bigg|\Big\{x:\Big|\sum_{\twoline{M^\theta\le s\le M}{s-\text{dyad.}}}\hams*\big(\sum_kb_k^{(s)}\big)(x)\Big|>0\Big\}\bigg|\le \\
  &\qquad\qquad\le\sum_k\sum_{\twoline{M^\theta\le s\le M}{s-\text{dyad.}}}\Big|\big\{x:\hams*b_k^{(s)}(x)\neq0\big\}\Big|\\
  &\qquad\qquad\le C\,\sum_k\sum_{\twoline{M^\theta\le s\le M}{s-\text{dyad.}}}\,s^{1/\alpha}\,\Big|\big\{x:|b_k(x)|>\lambda s^{1/\alpha}\big\}\Big|\\
  &\qquad\qquad= \frac{C}{\lambda}\,\sum_k\sum_{\twoline{M^\theta\le s\le M}{s-\text{dyad.}}}\,\lambda s^{1/\alpha}\,\Big|\big\{x:|b_k(x)|>\lambda s^{1/\alpha}\big\}\Big|\\
  &\qquad\qquad\le \frac{C}{\lambda}\,\sum_k\|b_k\|_{\ell^1}\\
  &\qquad\qquad\le \frac{C}{\lambda}\,\|f\|_{\ell^1},
\end{align*}
where the inequality before the last follows from the fact that the $s$'s are dyadic, while the last one follows from the disjointness of the supports.
We now turn to the component $B_k^{(s)}$. We use the following variant of Lemma 3.1 of \cite{PZ} (compare also Lemma 2.5 of \cite{UZ}).

\begin{lemma*} Suppose $M^\theta\le s_1\le s_2\le M$. Then for each $\delta>0$ there exists a $\gamma>0$ such that
\begin{gather}
  \mathcal H_{s_1}*\mathcal H_{s_2}(x)=\rho_{s_1,s_2}(x)\cdot\Big(\frac{s_1}{s_2}\Big)^\gamma+k_{s_1,s_2}(x)\qquad s_1<s_2,\\
  \hams*\hams(x)=\rho_{s,s}(x)+k_{s,s}(x)+\frac{1}{s^{1/\alpha}}\delta_0(x)\quad s=s_1=s_2,
\end{gather}
where
\begin{align}
  |\rho_{s_1,s_2}(x)|&\le \frac{C}{s_2}\\
  |\rho_{s_1,s_2}(x+h)-\rho_{s_1,s_2}(x)|&\le\frac{C}{s_2}\cdot\Big(\frac{|h|}{s_2}\Big)^\gamma\\
  |k_{s_1,s_2}(x)|&\le\frac{C}{s_2},\\
  {\rm supp}\,\rho_{s_1,s_2}&\subset[-Cs_2,Cs_2],\\
  {\rm supp}\,k_{s_1,s_2}&\subset[-Cs_2^{1-1/\alpha+\delta},Cs_2^{1-1/\alpha+\delta}].\label{kss_supp}
  \end{align}
\end{lemma*}
In fact, the term $k_{s_1,s_2}$ only appears when $s_1/s_2\sim1$. This observation follows
immediately by support considerations. The case $s_1=s_2$ was proved in \cite{UZ} and that proof
obviously extends to the general case. The support condition on $\rho$ was not mentioned explicitly
in \cite{UZ}, but it is obvious, considering the support of $\hams$. We have to estimate
\begin{equation*}
  \Big|\Big\{x:\Big|\sum_{\twoline{M^\theta\le s\le M}{s-\text{dyadic}}}\sum_k\hams*B_k^{(s)}(x)\Big|>\lambda\Big\}\Big|.
\end{equation*}
Note, that it is enough to restrict summation to $2^k<2s$. This is clear, since by support consideration
\begin{equation*}
  {\rm supp}\,\sum_{\twoline{s,k}{2^k\ge2s}}\hams*B_k^{(s)}\subset\bigcup_{j}Q^*_j,
\end{equation*}
and we can use the estimate \eqref{app:2}.
It then follows
\begin{align*}
  &\Big\|\sum_{\twoline{M^\theta\le s\le M}{s-\text{dyadic}}}\sum_{\twoline {k}{2^k\le2}}\hams*B_k^{(s)}\Big\|_{\ell^2}^2=\sum_{\twoline{M^\theta\le s_1,s_2\le M}{2^{k_1}\le s_1,2^{k_2}\le s_2}}\Big\langle \mathcal H_{s_1}*\mathcal H_{s_2}*B_{k_1}^{(s_1)},B_{k_2}^{(s_2)}\Big\rangle\\
  &\quad\qquad=\sum_{\twoline{M^\theta\le s\le M}{2^{k_1},2^{k_2}\le s}}\Big\langle \hams*\hams*B_{k_1}^{(s)},B_{k_2}^{(s)}\Big\rangle+\\
  &\qquad\qquad+2\Re\sum_{\twoline{M^\theta\le s_1<s_2\le M}{2^{k_1}\le s_1,2^{k_2}\le s_2}}\Big\langle \mathcal H_{s_1}*\mathcal H_{s_2}*B_{k_1}^{(s_1)},B_{k_2}^{(s_2)}\Big\rangle\\
  &\quad\qquad\le\sum_{\twoline{M^\theta\le s\le M}{2^{k_1},2^{k_2}\le s}}\Big(\Big|\Big\langle \rho_{s,s}*B_{k_1}^{(s)},B_{k_2}^{(s)}\Big\rangle\Big|+\Big|\Big\langle k_{s,s}*B_{k_1}^{(s)},B_{k_2}^{(s)}\Big\rangle\Big|+\frac{1}{s^{1/\alpha}}\,\big|\big\langle B_{k_1}^{(s)},B_{k_2}^{(s)}\big\rangle\big|\Big)+\\
  &\qquad\qquad+2\sum_{\twoline{M^\theta\le s_1<s_2\le M}{2^{k_1}\le s_1,2^{k_2}\le s_2}}\Big(\Big(\frac{s_1}{s_2}\Big)^\gamma\Big|\Big\langle \rho_{s_1,s_2}*B_{k_1}^{(s_1)},B_{k_2}^{(s_2)}\Big\rangle\Big|+\Big|\Big\langle k_{s_1,s_2}*B_{k_1}^{(s_1)},B_{k_2}^{(s_2)}\Big\rangle\Big|\Big)\\
\end{align*}
The estimate for each of the components is essentially the same. Let us consider
\begin{align*}
  \sum_{\twoline{M^\theta\le s\le M}{2^{k_1},2^{k_2}\le s}}\Big|\Big\langle \rho_{s,s}*B_{k_1}^{(s)},B_{k_2}^{(s)}\Big\rangle\Big|&=
  \sum_{\twoline{M^\theta\le s\le M}{2^k\le s}}\Big|\Big\langle \rho_{s,s}*B_k^{(s)},B_k^{(s)}\Big\rangle\Big|+\\
  &\quad+2\sum_{\twoline{M^\theta\le s\le M}{k_1<k_2,2^{k_2}\le s}}\Big|\Big\langle \rho_{s,s}*B_{k_1}^{(s)},B_{k_2}^{(s)}\Big\rangle\Big|.
\end{align*}
Again, the estimate for both sums is similar, and we outline the estimate for the first sum. By the definition, $B_k^{(s)}$ is supported on cubes $Q_j$ with $|Q_j|=2^k$
\begin{equation*}
  B_k^{(s)}=\sum_{j:|Q_j|=2^k}B_{k,Q_j}^{(s)}
\end{equation*}Thus
\begin{equation*}
  \rho_{s,s}*B_k^{(s)}(x)=\sum_{j:|Q_j|=2^k}\sum_{y\in\zet}\rho_{s,s}(x-y)B_{k,Q_j}^{(s)}(y).
\end{equation*}
Because of support considerations the outer sum can be restricted to these cubes, whose distance to $x$ is no greater than $Cs$
\begin{equation*}
  \rho_{s,s}*B_k^{(s)}(x)=\sum_{\twoline{j:\text{dist}(x,Q_j)\le Cs}{|Q_j|=2^k}}\sum_{y\in\zet}\rho_{s,s}(x-y)B_{k,Q_j}^{(s)}(y).
\end{equation*}
Each $B_{k,Q_j}^{(s)}$ has vanishing means so, denoting by $y_{Q_j}$ the center of cube $Q_j$,
\begin{align*}
  \big|\rho_{s,s}*B_k^{(s)}(x)\big| & = \Big|\sum_{\twoline{j:\text{dist}(x,Q_j)\le Cs}{|Q_j|=2^k}}\sum_{y\in\zet}\big(\rho_{s,s}(x-y)-\rho_{s,s}(x-y_{Q_j})\big)B_{k,Q_j}^{(s)}(y)\Big|\\
  &\le\sum_{\twoline{j:\text{dist}(x,Q_j)\le Cs}{|Q_j|=2^k}}\sum_{y\in\zet}\frac{C}{s}\,\Big(\frac{|y-y_{Q_j}|}{s}\Big)^\gamma\big|B_{k,Q_j}^{(s)}(y)\big|\\
  &\le C\,\sum_{\twoline{j:\text{dist}(x,Q_j)\le Cs}{|Q_j|=2^k}}\frac{1}{s}\,\Big(\frac{2^{k-1}}{s}\Big)^\gamma\big\|B_{k,Q_j}^{(s)}\big\|_{\ell^1}.
\end{align*}
By maximality of Calder\'on-Zygmund cubes
\begin{equation*}
  \big\|B_{k,Q_j}^{(s)}\big\|_{\ell^1}\le C\lambda|Q_j|.
\end{equation*}
Moreover, counting the number of dyadic cubes of size $2^k$ with distance to $x$ no greater than $Cs$ we can continue the above estimate
\begin{equation*}
  \le C\,\sum_{\twoline{j:\text{dist}(x,Q_j)\le Cs}{|Q_j|=2^k}}\frac{1}{s}\,\Big(\frac{2^{k-1}}{s}\Big)^\gamma\lambda |Q_j|\le C\,\Big(\frac{2^{k-1}}{s}\Big)^\gamma\lambda.
\end{equation*}
Thus
\begin{align*}
 \sum_{\twoline{M^\theta\le s\le M}{2^k\le s}}\Big|\Big\langle \rho_{s,s}*B_k^{(s)},B_k^{(s)}\Big\rangle\Big|&\le C\,\lambda
 \sum_{\twoline{M^\theta\le s\le M}{2^k\le s}}\Big(\frac{2^{k-1}}{s}\Big)^\gamma\big\|B_{k}^{(s)}\big\|_{\ell^1}\\
  & \le C\,\lambda\sum_{k\in\zet}\|f\|_{\ell^1(\cup_j\{Q_j:|Q_j|=2^k\})}\sum_{s\ge2^k}\Big(\frac{2^{k-1}}{s}\Big)^\gamma\\
  &\le C\,\lambda\,\|f\|_{\ell^1}.
\end{align*}
As already mentioned, all other components are estimated in a similar way. The functions
$k_{s_1,s_2}$ do not have the H\"older estimate, but the summation factor comes from the more
restricted support \eqref{kss_supp}, for the details see \cite{UZ}. Combining the above ideas, we
get
\begin{equation*}
  \Big\|\sum_{\twoline{M^\theta\le s\le M}{s-\text{dyadic}}}\sum_{\twoline {k}{2^k\le2}}\hams*B_k^{(s)}\Big\|_{\ell^2}^2\le
  C\,\lambda\,\|f\|_{\ell^1},
\end{equation*}
which, obviously, by the Chebychev's inequality implies the necessary estimate for the measure of the super-level set.

We now turn to the estimate of the last component $\sum E_k^{(s)}(f)$ in \eqref{app:1}. We use duality. Let $\|h\|_{\ell^2}=1$, and for $s\ge M^\theta$, dyadic let
\begin{equation*}
  \mathcal S_s=\sum_{\twoline{M^\theta\le \nu\le \min\{M,s\}}{\nu-\text{dyadic}}}\mathcal{H}_\nu
\end{equation*}
Summing by parts we have
\begin{align*}
&\Big|\Big<\sum_{\twoline{M^\theta\le s\le M}{s-\text{dyadic}}}\sum_k\hams*E_k^{(s)}(f),h\Big>\Big|^2\\
&\qquad=\Big|\Big<\sum_{\twoline{M^\theta\le s\le M}{s-\text{dyadic}}}\sum_k \mathcal S_s*(E_k^{(2s)}(f)-E_k^{(s)}(f)),h\Big>\Big|^2+\\
&\qquad\qquad+\Big|\Big<\sum_k \mathcal S_M*E_k^{(2M)}(f),h\Big>\Big|^2\\
&\qquad= I+II.
\end{align*}
We further estimate $I$. The $II$ part is estimated similarly.
\begin{align*}
I&=\Big|\sum_{\twoline{M^\theta\le s\le M}{s-\text{dyadic}}}\Big<\sum_k \sum_{j:|Q_j|=2^k}\frac{\mathbbm{1}_{Q_j}}{|Q_j|}\int_{Q_j} \big(\mathbbm{1}_{\{|f|\le\lambda (2s)^{1/\alpha}\}}-\mathbbm{1}_{\{|f|\le \lambda s^{1/\alpha}\}}\big)f,\mathcal S_s*h\Big>\Big|^2\\
&=\Big|\sum_{\twoline{M^\theta\le s\le M}{s-\text{dyadic}}}\Big<\sum_k \sum_{j:|Q_j|=2^k}\frac{\mathbbm{1}_{Q_j}}{|Q_j|}\int_{Q_j\cap\{\lambda s^{1/\alpha}<|f|\le \lambda (2s)^{1/\alpha}\}} f,\mathcal S_s*h\Big>\Big|^2\\
&\le\Big<\sum_{\twoline{M^\theta\le s\le M}{s-\text{dyadic}}}\sum_k \sum_{j:|Q_j|=2^k}\frac{\mathbbm{1}_{Q_j}}{|Q_j|}\int_{Q_j\cap\{\lambda s^{1/\alpha}<|f|\le \lambda (2s)^{1/\alpha}\}} |f|,\sup_{s-\text{dyadic}}|\mathcal S_s*h|\Big>^2\\
&\le\Big<\sum_k \sum_{j:|Q_j|=2^k}\frac{\mathbbm{1}_{Q_j}}{|Q_j|}\int_{Q_j} |f|,\sup_{s-\text{dyadic}}|\mathcal S_s*h|\Big>^2\\
&\le C\Big\|\sum_kE_k(|f|)\Big\|^2_{\ell^2},
\end{align*}
where in the last inequality we have used the strong $(2,2)$ type of
\begin{equation*}
  S^*(h)=\sup_{s-\text{dyadic}}|S_s*h|
\end{equation*}
which is well known as a folklore. It follows, since this operator can be controlled by a maximal,
with respect to truncations, CZ operator (in the sense of this paper), and a square function.
Further we observe
\begin{align*}
  \Big\|\sum_kE_k(|f|)\big\|_{\ell^2}^2 &=\sum_k\big\|E_k(|f|)\big\|_{\ell^2}^2\\
   &\le C\lambda\sum_k\big\|E_k(|f|)\big\|_{\ell^1},
\end{align*}
where the first equality follows since $E_k(|f|)$ have disjoint supports and in the last inequality
we have used the maximality of Calder\'on-Zygmund cubes used in the definition of $E_k(|f|)$. Consequently,
\begin{equation*}
  \Big\|\sum_kE_k(|f|)\Big\|^2_{\ell^2}\le C\lambda\|f\|_{\ell^1}.
\end{equation*}
This is the required estimate of the last component in the representation \eqref{app:1}, and thus the proof of the theorem is concluded.
\end{proof}

\end{document}